\documentclass[reqno,a4paper,11pt]{amsart}

\usepackage{amsmath,amsthm,amsfonts,amssymb,latexsym}
\usepackage{verbatim, graphicx, ifthen, color}
\usepackage[T1]{fontenc}
\usepackage[applemac]{inputenc}

\newtheorem{thm}{Theorem}[section]
\newtheorem{prop}[thm]{Proposition}
\newtheorem{lem}[thm]{Lemma}
\newtheorem{cor}[thm]{Corollary}

\theoremstyle{definition}
\newtheorem{defn}[thm]{Definition}
\newtheorem{example}[thm]{Example}
\newtheorem{rem}[thm]{Remark}

\definecolor{gorange}{rgb}{1,0.5,0}
\definecolor{myblue}{rgb}{0,0.33,0.55}
\definecolor{mygreen}{rgb}{0,0.61,0.04}
\definecolor{mybrown}{RGB}{182,125,17}

\renewcommand{\qed}{\hfill $\blacksquare$}
\def\bdem{\begin{proof}}
\def\edem{\end{proof}}

\newcommand{\peso}[1]{ \quad \text{ \rm  #1 } \quad }
\newcommand{\sub}[2]{{#1}_{\mbox{\tiny{${#2}$}}}}

\def\fii{\phi }

\def \Z {\mathbb{Z}}
\def \C {\mathbb{C}}
\def \N {\mathbb{N}}
\def \R {\mathbb{R}}

\def \E {\mathbb{E}}

\def\eme{\mathcal{M}}

\renewcommand{\Re}{\operatorname{Re}}
\renewcommand{\Im}{\operatorname{Im}}

\newcommand{\abs}[1]{|#1|}
\newcommand{\Abs}[1]{\left|#1\right|}

\newcommand{\inner}[2]{\left\langle #1,#2 \right\rangle}

\newcommand{\Ai}{\mathrm{Ai}}
\newcommand{\norm}[1]{\|#1\|}
\newcommand{\dif}{\mathrm{d}}
\newcommand{\e}{\mathrm{e}}
\newcommand{\im}{\mathrm{i}}

\begin{document}

\title{Zeros of random functions generated with de Branges kernels}

\author{Jorge Antezana\and Jordi Marzo \and Jan-Fredrik Olsen}

\address{Universidad Nacional de La Plata, Departamento de Matem\'atica, Esq. 50 y 115 s/n, Facultad de Ciencias Exactas de La Plata (1900), Buenos Aires, Argentina}
\email{antezana@mate.unlp.edu.ar}

\address{Departament de matem\`atica aplicada i an\`alisi, Universitat de Barcelona, Gran Via 585, 08007, Barcelona, Spain}
\email{jmarzo@ub.edu}

\address{Centre for Mathematical Sciences, Lund University, P.O. Box 118, SE-221 00 Lund, Sweden}
\email{janfreol@maths.lth.se}


\date{\today}
\begin{abstract}

We study the point process given by the set of real zeros of random sums of orthonormal bases of reproducing kernels of de Branges spaces. 
Examples of these kernels are the cardinal sine, Airy and Bessel kernels.
We find an explicit formula for the first intensity function in terms of the phase of the Hermite-Biehler function.  
We prove that the first intensity of the point process completely characterizes the underlying de Branges space. This result is
a real version of the so called Calabi rigidity for GAFs proved by M.\,Sodin.
\end{abstract}

\maketitle


\section{Introduction}

\subsection{Background}
 Let $E(z)$ be a function of the Hermite-Biehler 
 class, i.e., $E(z)$ is entire and satisfies the inequality $\abs{E(z)} > \abs{E^\ast(z)}$ for $\Im z > 0$, where $E^\ast(z) = \overline{E(\bar{z})}$. Given such a function,
 the de Branges space $H(E)$ is defined by
\begin{equation*}
	H(E) = \left\{  f \ \text{entire}   :     \frac{f}{E}, \frac{f}{E^\ast} \in H^2(\C_+) \right\},
\end{equation*}
with norm given by 
\begin{equation*}
	 \| f \|_{H(E)}^2=\int_{-\infty}^{+\infty} \Abs{\frac{f(x)}{E(x)}}^2 \dif x.
\end{equation*}
Here, we use $H^2(\C_+)$ to denote the Hardy space of the upper half-plane. The standard reference for de Branges spaces is the book \cite{dB1968} by de     Branges. 

Note  that in  this paper, we restrict ourselves to considering Hermite-Biehler functions without real zeroes. This excludes the existence of points $x \in \R$ such that $f(x) = 0$ for all $f \in H(E)$.

The spaces $H(E)$ can be thought of as weighted versions of the Paley-Wiener spaces $PW_a^2$. The simplest description  of $PW_a^2$ is perhaps as consisting of the Fourier transforms of functions in $L^2(-a,a)$ 
with the induced norm of $L^2(\R).$ By the Paley-Wiener theorem, this space is identical to $H(E)$ with $E(z)= \e^{-\im a z}$.   And, as we point out below, its reproducing kernel is given by translates of the cardinal sine function, i.e., 
\begin{equation*}
	K_y(x) = \frac{\sin a (x-y)}{(x-y)}.
\end{equation*}

Other important examples
  are obtained    by using the Airy or Bessel function. 
  The corresponding reproducing kernels, together with the cardinal sine, 
  are ubiquitous in random matrix theory. 
  A reason is that the determinantal random point processes given by these kernels   (see below for a definition) 
  describe the possible asymptotic distribution of eigenvalues from large Hermitian matrices, see, for instance, \cite{Blower}.

 In this paper, instead of    considering  determinantal processes,   we investigate random point processes whose relation to the cardinal sine, Bessel and Airy functions   is slightly different. Namely, translates of these functions essentially  appear as orthonormal bases of reproducing kernels of certain de Branges spaces. We   study the point processes obtained by considering the random linear combinations obtained from each of these bases. These are examples of Gaussian Analytic Functions (GAFs).

\begin{rem}\label{relationModel}  
The de Branges spaces can be seen as a subclass of   model subspaces of the Hardy space $H^2(\C_+).$
  The model subspaces $\eme(\Theta)$ of $H^2$ are exactly those of the form   $\eme(\Theta) = H^2\cap ( \Theta H^2)^\perp$, 
  where $\Theta$ is  an inner function. That is, $\Theta$ is bounded on the upper half-plane and satisfies $\abs{\Theta} =1$ on $\R$. To see the connection, one may easily verify that if $E$ is of Hermite-Biehler class, then  $\Theta = E^\ast/E$ is a meromorphic inner function, and, conversely, it is possible to show that any meromorphic inner function can be factorized in this way. In other words,  the operator $f\mapsto f/E,$ where $\Theta=E^*/E,$ is an isometry from $H(E)$ onto $\eme(\Theta)$. See \cite{havinmashreghi2003a} for more details.
\end{rem}

 To introduce the GAFs that are our object of study,  we first  recall some additional facts and notations about the spaces $H(E)$.
For a point $w \in \C$, the reproducing kernel function $K_w \in H(E)$ is the function satisfying 
\begin{equation*}
	f(w) = \inner{f}{K_w} = \int_{-\infty}^{+\infty} f(x)\overline{K_w(x)}\frac{\dif x}{|E(x)|^{2}}
\end{equation*}
for all $f \in H(E)$. This function exists since point evaluations are bounded functionals. A straight-forward computation yields the formula
\begin{equation*}
K_w(z)=K(z,w)=\frac{\im }{2\pi}\frac{E(z)\overline{E(w)}-E^*(z)\overline{E^*(w)}}{z-\overline{w}}.
\end{equation*}

It is useful to introduce the polar decomposition $E(x)  = \abs{E(x)} \e^{- \im \phi(x)}$ for $x\in \R.$ 
Here, the so called phase function $\phi$ is an increasing $C^\infty(\R)$ function. With this notation, for $x \in \R$, we have
$$
\|K_x\|^2=K(x,x)=\frac{1}{\pi}\phi'(x)|E(x)|^2, 
$$
and, for $x,y \in \R$, we get the following expression for the normalized reproducing kernels 
\begin{equation*}
	k_y(x) = \frac{K(x,y)}{K(y,y)^{1/2}} = \frac{ \abs{E(x)} }{\sqrt{\pi\phi'(y)}} \frac{  \sin( \phi(x) - \phi(y))}{ x-y}.
\end{equation*}

Observe that $k_y(x) = \langle k_y,k_x \rangle = 0$ whenever $x,y \in \R$ are such that $\phi(x)-\phi(y) = k \pi$ for $k \in \Z$. 
This means that if $\{\omega_n\}$ is the sequence of points such that $\phi(\omega_n) = \alpha + \pi n$, for some $\alpha \in [0,\pi)$, then the family 
of functions $\{ k_{\omega_n}\}$ forms an orthonormal system in $H(E)$. 
However, as de Branges proved, see \cite[p. 55]{dB1968}, more is true. 
\begin{lem} \label{alpha-lemma}  
	Suppose that $H(E)$ is a de Branges space. Then, for all $\alpha \in [0,\pi)$, 
except at most  one, the system $\{ k_{\omega_n}\}$
is an orthonormal basis for $H(E)$.  Moreover, these are the only orthonormal bases of reproducing kernels, and the exceptional $\alpha$ is characterized by the condition that  $\e^{\im \alpha z}E(z)-\e^{-\im\alpha z}E^*(z)\in H(E)$.  
\end{lem}

We now give the definition of a de Branges GAF.
 \begin{defn}  \label{definition_deBranges_GAF}
Let $H(E)$ be a de Branges space with phase function $\phi(x)$ and an orthonormal basis of reproducing kernels $\{k_{\omega_n}\}$. For real i.i.d.\,standard normal random variables $a_n$, we call the random function
\begin{equation}
\begin{aligned} 						\label{de Branges GAF} 
F(x)=\sum_{n} a_n k_{\omega_n}(x)&= \sum_n a_n \frac{\im }{2}\frac{E(x)\overline{E(\omega_n)}-E^*(x)\overline{E^*(\omega_n)}}{\sqrt{\pi\phi'(x)}\abs{E(x)}(x-\omega_n)}  
\\ &=\sum_n a_n \frac{ \abs{E(x)} }{\sqrt{\pi\phi'(\omega_n)}} \frac{  \sin \big( \phi(x) - \phi(\omega_n)\big)}{ x-\omega_n} 
\end{aligned}
\end{equation}
  a de Branges GAF.
\end{defn}

We remark that Kolmogorov's inequality shows that, for a fixed $x\in\R$, the series (\ref{de Branges GAF}) is almost surely pointwise
  convergent. In fact, since $\sum_n | k_{\omega_n}(z)  |^2$
  converges uniformly for $z$ on compact subsets of $\C$, it can be proved that the series in \eqref{de Branges GAF} almost surely converges uniformly on compact subsets of $\C$. Thus, the de Branges GAF is   an entire function with probability one. 

More generally, for a sequence of analytic function $\{ \psi_n(z) \}$ defined in a certain region $\Lambda\subset \C,$
  if $\sum_n |\psi_n(z)|^2$ converges uniformly on compact sets of $\Lambda$ and 
  $a_n$ are  real  i.i.d.\,standard random variables,
  the series 
\begin{equation} \label{general GAF}  \Psi(z)=\sum_n a_n \psi_n(z) \end{equation}
   defines an analytic function with probability one. 
  If $a_n$ are complex i.i.d.\,standard normal  random variables, the function $\Psi(z)$ is called a Gaussian Analytic Function (GAF),  see
  \cite[Lemma 2.2.3]{hough_krishnapur_peres_virag2009}.

  The  main feature connecting the zeros of the de Branges GAF $F(x)$ with the de Branges space $H(E)$ is that, by definition,
  the 
  covariance kernel of the gaussian stochastic process $(F(x))_{x\in\R}$ 
  coincides with the reproducing kernel of the de Branges space $H(E),$ i.e.,
$$\E[F(x)F(y)]=K(x,y).$$
  Here $\E$ denotes the expectation. Recall that the covariance kernel, as $(F(x))_{x\in\R}$ has mean zero, completely defines the behavior of the process.

\begin{rem} 
	By Remark \ref{relationModel} above, the random function $F/E$ should be called a model space GAF. Since $E(x)$ is a non-random function without zeros on the real line, the zeros of $F/E$ has exactly the same distribution as that of $F$. We will therefore pass freely between the two. Note that if we factor out non-random factors, then we can write
	\begin{equation*}
		\frac{F(x)}{ E(x)} = \frac{ \e^{\im \phi(x)} \sin(\phi(x)-\alpha)}{\sqrt\pi}\sum_{n} a_n \frac{(-1)^n}{\sqrt{\phi'(\omega_n)}(x-\omega_n)},
	\end{equation*}
	where $\alpha \in \R$ is the parameter so that $\{ \omega_n \}$ is the sequence of points such that $\phi(\omega_n) = \alpha + \pi n$.
\end{rem}

To formulate our results, we need to define the intensity functions, see \cite[Chap. 1]{hough_krishnapur_peres_virag2009}, which 
describe the distribution of a point process and, in some sense, are similar to densities.
%
To this end, suppose that  a simple point process $\mathcal{X}$ is defined in $\R$ or $\C$, i.e., the points of $\mathcal{X}$ are almost surely of multiplicity one. In our case, this process will be the real zeroes of a de Branges GAF (\ref{de Branges GAF}). 
The joint intensities $\rho_k(x_1,\dots , x_k)$ are defined by the relations
 $$\E[\mathcal{X}(D_1)\cdots \mathcal{X}(D_k) ]=\int_{D_1\times \dots \times D_k} \rho_k(x_1,\dots , x_k) \dif x_1\dots \dif x_k,$$
  where $\mathcal{X}(A)$ stands for the number points of $\mathcal{X}$ in a Borel set $A\subset \R,$ the Borel subsets $D_1,\dots, D_k\subset \R$ are mutually 
  disjoint, and $\rho_k(x_1,\dots , x_k)=0$ 
  if $x_i=x_j$ for $i\neq j.$ Other reference measures can be considered, but throughout this work we will consider only intensity functions with respect
  to the Lebesgue measure.

In particular, if $Z_\R(F)$ denotes the set of real zeroes of a de Branges GAF 
$$F(x)=\sum_n a_n k_{\omega_n}(x),$$
its first intensity satisfies 
$$\E[\# (Z_\R(F)\cap D)]=\int_D \rho_1(x)\dif x.$$

  A particularly   well studied class of point processes is defined by assuming that the intensities are given by determinants 
  of certain kernel functions
 $$\rho_k(x_1,\dots , x_k)=\det (K(x_i,x_j))_{i,j=1,\dots, k}.$$
 These are exactly the deteminantal point processes mentioned above, see \cite{Blower,Soshnikov2000}. 
 The de Branges reproducing kernels provide some well known examples of determinantal point processes.
 These processes have been studied specially in connection to random matrix theory.
 More specifically, the cardinal sine kernel describes the distribution of neighbouring 
eigenvalues in the bulk for the Gaussian unitary ensemble, GUE \cite{mehta1991},
 the Airy kernel (\ref{airykernel}) describes the soft edge of 
the spectrum for the GUE \cite{TracyWidom1994a}, while  the Bessel kernel (\ref{besselkernel})
 represents the hard edge of the spectrum of the Jacobi and Laguerre ensembles, \cite{TracyWidom1994b}.

In  the determinantal case, the first intensity function is just given by the diagonal of the kernel, namely  $\rho_1(x) = K(x,x)$. For GAFs defined as in \eqref{general GAF} with the $a_n$ being complex i.i.d.\,standard normal variables,
  Sodin \cite{sodin2000}, computed that
\begin{equation}						\label{rho1GAF}
\rho_1(z) = \frac{1}{\pi}\Delta \log K(z,z)
\end{equation}
whence  $$\E[\#(Z_\C(\Psi)\cap D)]=\frac{1}{\pi} \int_D \Delta \log K(z,z) dz$$
  where $D\subset \Lambda$ and $Z_\C(\Psi)$ is the zero set of $\Psi(z)$. 
  This generalises previous works of Kac \cite{kac1943}, Rice \cite{rice1939, rice1945}, and Edelman-Kostlan     \cite{edelman_kostlan1995}.

 In  our case, as we consider real random coefficients, and not complex, the first intensity cannot be computed using the   expression \eqref{rho1GAF}. 
Instead, the situation is more similar to the case of  real zeroes of random polynomials with real coefficients,  studied in \cite{TaoVu13}.  
 Indeed, in the case of real random coefficients, the de Branges GAFs have zeroes which are symmetric with respect to the real line (recall that the kernel functions are  symmetric with respect to the real line), and   one expects to find a certain proportion of them on the real line.
  For this reason, with respect to the complex plane, the  intensity functions should have a singular part that is supported on the real line. Indeed, this is a consequence of the   following formula which was proved in 
  \cite{feldheim2011}. It says that in the distributional sense, the first intensity function satisfies
\begin{equation}			\label{rho1sym}
\rho_1(z) = \frac{1}{\pi} \Delta \log \left( K(z,z)+ \sqrt{K(z,z	)^2-|K(z,\bar{z})|^2}  \right). 
\end{equation}

\subsection{Results}   

In the following theorem, we give a formula for the first intensity function of de Branges GAFs in terms of the   Schwarzian derivative of the phase function $\phi$, i.e.,
 $$S \phi=\Big(\frac{\fii''}{\fii'}\Big)'-\frac{1}{2}\Big(\frac{\fii''}{\fii'}\Big)^2.$$ 
\begin{thm} \label{first intensity}  
Let $H(E)$ be a de Branges space with phase function $\phi$. For any choice of orthonormal basis of reproducing kernels, then the first intensity function of the corresponding de Branges GAF satisfies
\begin{equation*}
	\rho_1(x) =  \frac{1}{\pi} \sqrt{  \frac{\phi'(x)^2 }{3} + \frac{S \phi(x)}{6}  }.
\end{equation*}
\end{thm}
	
\noindent To compute $\rho_1(x)$ one can use any of the following equivalent formulas:
	\begin{align*}
		\rho_1(x) &= \mathbb{E}\big[  \abs{F'(x)}  \ : \  F(x) = 0 \big] &&   \text{(Rice formula)}  \\
		&= \frac{1}{\pi} \sqrt{\frac{\partial^2}{\partial t \partial s} \log K(s,t)\Big|_{s=t=x}  }  && \text{(Edelman-Kostlan)} \\
		&= \frac{ \sup \big\{  \abs{h'(x)} :  \norm{h}_{H(E)} =1, \ h(x)=0   \big\}}{\pi \sqrt{K(x,x)}}  && \text{(Bergman metric)},
	\end{align*}
	where   $\mathbb{E}\big[  \abs{F'(x)}  \, : \,  F(x) = 0 \big]$ denotes a conditional expectation.

  We found the appearance of the Schwarzian derivative, which is an infinitesimal version of the coss-ratio (see \cite{Ahlfors1988})
quite intriguing.  
It is invariant under any Moebius transformation and it measures, in some sense, how far a map is from being a Moebius map.
Geometrically, it is connected with   curvature.
It appears in the study of univalent functions, generalizations of Schwarz-Christoffel mappings, Sturm-Liouville equations and real dynamics, 
see the survey paper \cite{osgood98}.

\begin{example}
The  Paley-Weiner space $PW_a^2$ is a de Branges space with $E(z) = \e^{- \im a z}$. It follows that $\phi(x) = ax$, and that the Paley-Wiener GAF is given by
$$F(x) = \sum_{n\in \Z} a_n \frac{\sin a (x-n)}{a (x-n)}.$$
In this case, the process given by the real zeroes is stationary.  
By Theorem \ref{first intensity}, it follows   that $\rho_1(x) \equiv a/\pi\sqrt{3}$. Although well-known, we point out that  this is surprising 
since it implies that the expected number of zeroes to be found in an interval  is strictly smaller  than what one would expect 
by comparing to the density of the zeroes of the reproducing kernel. Some properties of this 
process where studied in \cite{antezana_buckley_marzo_olsen}.
\end{example}

Below, in Section \ref{example section}, we will consider examples of de Branges spaces given by the Airy and Bessel functions. Here, we state the following lemma which will be used to transfer information of the derivative of the phase function $\phi'(x)$ to the first intensity function $\rho_1(x)$.
\begin{lem} \label{schwarzian lemma}
	Suppose that $\phi(x)$ is a function such that
	\begin{equation*}
		\phi'(x) = \frac{C}{x^\alpha} +  g(x) 
	\end{equation*}
	where
$g(x) = o ( x^{-\alpha} )$, $g'(x) = o ( x^{-\alpha-1})$ and $g''(x) = o ( x^{-\alpha-2}  )$ as $x \rightarrow \infty$, respectively. Then, as $x \rightarrow \infty$,
\begin{equation*}
	S\phi(x) =   \frac{2\alpha(1-\alpha)}{x^2}  + o\Big( \frac{1}{x^2}   \Big).
\end{equation*}
\end{lem}
\begin{proof}
	The proof follows from a direct computation using the definition of $Sg$.
\end{proof}


  An interesting feature of GAFs is that the distribution of the zero set depends just on the first intensity function,
  as was shown by Sodin in \cite{sodin2000}  (see also \cite[Section 2.5.]{hough_krishnapur_peres_virag2009}). 
  The idea is that from (\ref{rho1GAF}), it can be shown that the difference between the logarithms of the covariance kernel of two GAFs 
  with the same first intensity is an harmonic function. This implies that the covariance kernels are equal on the complex diagonal,  up to a multiplication by a non-vanishing harmonic function. From this, one is able to argue that the kernel functions are essentially equal even off the diagonals, which leads to the desired  result. 

  As  Sodin mentioned in his paper, his result can be seen as a special case of 
  Calabi's rigidity theorem, \cite{Cal53,Law71}. A version in the setting of model subspaces was proved by Nikolskii, \cite{Nik95}. 
  In our case, as we are working in a real context, the formula \eqref{rho1GAF} is no longer valid. However,  as we show in the following theorem, the first intensity still determines the distribution.

\begin{thm} 					\label{rigidity_1}
  Let $F(x),G(x)$ be two de Branges GAFs and let $\rho_1^{F}(x),\rho_1^{G}(x)$ be the  respective first intensity functions.
   If, for all $x \in \R$,
$$\rho_1^{F}(x)=\rho_1^{G}(x),$$
 then there exists a non-random analytic function $S(z),$ 
which does not vanish anywhere,
  such that $S(x)F(x)$ and $G(x)$ have the same distribution.
  In particular,
  the random processes given by the zeros of $F(x)$ and $G(x)$ have the same distribution.
\end{thm}


This result follows as a consequence of the following theorem, which shows that the first intensity function not only determines the de Branges GAF, but also the underlying de Branges space up to a very special isometric isomorphism.

\begin{thm}\label{rigidity_2}
 Let $E_1(z),E_2(z)$ be Hermite-Biehler functions with phase functions respectively given by $\phi_1(x),\phi_2(x)$ and let
 $K_1(z,w),K_2(z,w)$ reproducing kernels of the de Branges spaces $H(E_1),H(E_2).$
 If
 $$\rho_1^{E_1}(x)=\rho_1^{E_2}(x),\;\;\mbox{for all}\;\;x\in \R,$$ 
then there exists an entire function $S(z)$ without zeros, such that $$K_2(z,w)=S(z)K_1(z,w)\overline{S(w)},$$
  and the map $f\mapsto Sf$ is an isometry from $H(E_1)$ to $H(E_2),$
\end{thm}

  Observe that it follows from this result that a de Branges GAF is not stationary unless the kernel is essentially the reproducing kernel of a Paley-Wiener space, (see \cite[Corollary 2.5.4.]{hough_krishnapur_peres_virag2009}).

\subsection{Structure of the paper}   In Section 2, we prove Theorem \ref{first intensity}. In Section 3, we apply Theorem \ref{first intensity} to several examples, including de Branges spaces with the cardinal sine, Bessel and Airy functions yielding the reproducing kernel as well as to de Branges spaces connected to orthogonal polynomials. In Section 4, we prove Theorem \ref{rigidity_1} and
Theorem \ref{rigidity_2}.

\subsection{Notation} We use the notation $f(x) \lesssim g(x)$  to 
indicate that $f(x)/g(x)$ is bounded above by some positive constant. We write $f(x) \simeq g(x)$ if both $f(x) \lesssim g(x)$ and $g(x) \lesssim f(x)$ hold.


\section{The first Intensity function} \label{Section 2}

In this section, we explain two methods for computing the first intensity function of a de Branges GAF. This proves Theorem \ref{first intensity}. The first method relies more explicitly on probability theory, and is based on the  Rice formula (see formula \eqref{rice formula} below). The second method is function theoretic and is based on solving a deterministic maximal problem. 
This maximal problem is arrived at via the classical Edelman-Kostlan formula and its connection to the Bergman metric.  

\subsection{Computation of $\rho_1(x)$ using Rice's formula}

As we consider Hermite-Biehler functions $E(z)$ without zeros on $\R,$ the process given by the zeros of the de Branges GAF  $F(x)$, defined by \eqref{de Branges GAF}, has the same distribution as the process given by the zeros of 
$$
\sub{F}{\Theta}(x)=\frac{F(x)}{E(x)}=\e^{\im \phi(x)}\sum_n a_n \frac{\sin (\phi(x)-\phi(\omega_n))}{\sqrt{\pi \phi'(\omega_n)}(x-\omega_n)} :=\sum_n a_n k_\Theta (x,\omega_n),
$$
where $k_\Theta (x,\omega_n)$ is the normalized reproducing kernel in $\eme(\Theta)$ (see Remark \ref{relationModel}).
Observe that this corresponds to considering the GAF defined by an orthonormal basis of reproducing kernels of the model subspace $\eme(\Theta).$

According to the Rice formula (see \cite[Chap. 11]{AT09}), the first intensity of the point process associated to the real zeroes process $\sub{F}{\Theta}(x)$ is given by 
\begin{equation} \label{rice formula}
	\rho_1(x) = \mathbb{E}\left[ \abs{\sub{F}{\Theta}'(x)} : \sub{F}{\Theta}(x)=0 \right] =  \int_\R  \abs{t}\, p_{\sub{F}{\Theta}(x),\sub{F}{\Theta}'(x)}(0,t) d t,
\end{equation}
where $p_{\sub{F}{\Theta}(x),\sub{F}{\Theta}'(x)}$ is the joint probability density of the two dimensional normal vector $(\sub{F}{\Theta}(x), \sub{F}{\Theta}'(x))$. Since both $\sub{F}{\Theta}(x)$ and $\sub{F}{\Theta}'(x)$ are normally distributed, it follows that 
\begin{equation*}
	 p_{\sub{F}{\Theta}(x),\sub{F}{\Theta}'(x)}(s,t) = \frac{1}{2\pi \sqrt{\abs{\Sigma}}} \exp\Big({\frac{1}{2\abs{\Sigma}} (s,t) \Sigma^{-1} (s,t)^T}\Big),
\end{equation*}
where the covariance matrix $\Sigma$ is given by
\begin{equation*}
	\Sigma = \left( \begin{matrix} \mathbb{E}\left[\sub{F}{\Theta}(x)^2\right] & \mathbb{E}\left[\sub{F}{\Theta}(x)\sub{F}{\Theta}'(x)\right] \\ \mathbb{E}\left[\sub{F}{\Theta}(x)\sub{F}{\Theta}'(x)\right] & \mathbb{E}\left[\sub{F}{\Theta}'(x)^2\right]  \end{matrix} \right)  := \left( \begin{matrix} a & b \\ c & d \end{matrix} \right).
\end{equation*}
  Let $\sub{K}{\Theta}(x,y)$ be the reproducing kernel in $\eme (\Theta).$ We compute the entries: 
\begin{align*}
	&& a = \mathbb{E}\left[\sub{F}{\Theta}(x)^2\right]  &=\sum_{n} k_\Theta (x,\omega_n)^2  \\ 
	&& &= \sub{K}{\Theta}(x,x) =  \frac{\phi'(x)}{\pi}, \\
	&& d =\mathbb{E} \left[\sub{F}{\Theta}'(x)^2\right] &=\sum_{n} \Big( \frac{\partial}{\partial x} k_\Theta (x,\omega_n) \Big)^2 \\
	&& &= \frac{\partial^2}{\partial s \partial t}  \sub{K}{\Theta}(s,t)\Big|_{s=t=x}  =\frac{2\phi'(x)^3+\phi'''(x)}{6\pi},
\end{align*}
and
\begin{align*}
	 && b = c =  \mathbb{E}\left[\sub{F}{\Theta}(x)\sub{F}{\Theta}'(x)\right]  &=  \sum_{n} k_\Theta (x,\omega_n) \frac{\partial}{\partial x} k_\Theta (x,\omega_n) \\
	&&  &=\frac{1}{2} \frac{\partial}{\partial x} \sum_n k_\Theta (x,\omega_n)^2 \\ && &= \frac{1}{2}\frac{\partial}{\partial x} \sub{K}{\Theta}(x,x)=\frac{\phi''(x)}{2\pi}.
\end{align*}
On the other hand, since
\begin{equation*}
	\Sigma^{-1} = \frac{1}{\abs{\Sigma}} \left( \begin{matrix} d & -b \\ -c & a \end{matrix} \right),
\end{equation*}
it follows that
\begin{align*}
	\rho_1(x) &= \frac{1}{2\pi \abs{\Sigma}^{1/2}} \int_\R \abs{t} \e^{-\frac{1}{2} (0,t) \Sigma^{-1} (0,t)^T } \dif t  \\ &= \frac{1}{\pi \abs{\Sigma}^{1/2}} \int_0^\infty t \e^{-\frac{ a t^2}{2\abs{\Sigma}} } \dif t \\
	&=  \frac{\abs{\Sigma}^{1/2}}{\pi a}.
\end{align*}
As $\abs{\Sigma} =  ad - bc$, this yields
\begin{align*}
 \rho_1(x) &= \frac{1}{\pi} \sqrt{\frac{1}{a} \Big(d - \frac{bc}{a}\Big)}  \\
 &
 =\frac{1}{\pi} \sqrt{ \frac{\pi}{\phi'(x)}\left( \frac{2(\phi'(x))^3+\phi'''(x)}{6\pi} -\frac{\pi}{\phi'(x)}\frac{(\phi''(x))^2}{4\pi^2}\right) } \\
 &=  
  \frac{1}{\pi} \sqrt{\frac{\phi'(x)^2}{3}  + \frac{S\phi(x)}{6}},
\end{align*}
where, as before, $S\phi$ denotes the Schwarzian derivative of $\phi.$ 
\qed

\subsection{Computation of   $\rho_1(x)$ using the Bergman metric}
 To compute the first intensity function, one can alternatively use the Edelman-Kostlan formula, \cite{edelman_kostlan1995}. It says that for a real point process with covariance 
 kernel $K(s,t)$, it holds that
 \begin{equation} \label{edelman-kostlan}
 	\rho_1(x) = \frac{1}{\pi} \sqrt{ \frac{\partial^2}{\partial t \partial s} \log K(t,s)\Big|_{t=s=x}}.
 \end{equation}

Following Bergman \cite[p. 35, formula (27)]{Bergman}, one can relate this expression with an extremal problem
\begin{equation*} 
	\left.\frac{\partial^2}{\partial t \partial s} \log K(t,s)\right|_{s=t=x} = 
	K(x,x)\frac{    \inf     \{ \| h\|_{H(E)}^2\;:\;  \; h(x)=1 \}^2}{\inf\{ \| h\|_{H(E)}^2\;: \;   \;h(x)=0,\;h'(x)=1 \}},
\end{equation*}
where it is implicit that both infimums are taken only over $h \in H(E)$.     This can be reformulated as
\begin{equation} \label{bergman}
	\left.\frac{\partial^2}{\partial t \partial s} \log K(t,s)\right|_{s=t=x}=\frac{\sup \{ |h'(x)|^2 : \|h\|_{H(E)}=0,\;h(x)=1  \}}{  K(x,x)}.
\end{equation}
  This last expression is related with the Bergman metric for the space ${H}(E)$ (see e.g., \cite[Ch. 6]{JP} and \cite{ARSW11}). 
  Indeed, the first intensity
  $\rho_1(x)$ is the density, with respect to the real line, of the Bergman metric given by the reproducing kernel of $H(E).$ Using \eqref{bergman}, we can compute $\rho_1(x)$ by solving this (deterministic) extremal problem in ${H}(E)$. 
To this end, we prove the following technical lemma. Just as a curiosity, this lemma provides a generalisation of the Basel problem.

\begin{lem}\label{serie 2}
Let $H(E)$ be a de Branges space with phase function $\phi$. If  $\{k_{\omega_n}\}$ is an orthonormal basis of reproducing kernels for $H(E)$, then 
\begin{equation}              \label{serie en ell2}
\sum_{n\neq k} \frac{\fii'(\omega_k)}{(\omega_k-\omega_n)^2\fii'(\omega_n)}=
\frac{\fii'(\omega_k)^2}{3}+\frac{S\phi(\omega_k)}{6}.
\end{equation}
\end{lem}
\bdem
Fix $n$, and let $x\in \R\setminus \{\omega_n\}.$ Computing the norm of the normalized reproducing kernel $k_x$, we get
\begin{equation*}					\label{boludez}
\sum_{n}\frac{1}{\fii'(\omega_n)(x-\omega_n)^2}=\frac{\phi'(x)}{\sin^2(\alpha-\phi(x))}\,,
\end{equation*}
whence
\begin{align*}
\sum_{n\neq k}\frac{\fii'(\omega_k)}{\fii'(\omega_n)(x-\omega_n)^2}&=\frac{\fii'(\omega_k)\fii'(x)}{\sin^2(\fii(\omega_k)-\fii(x))}-\frac{1}{(x-\omega_k)^2}\,
\\&=\left(\frac{\fii'(\omega_k)\fii'(x)}{\sin^2(\fii(\omega_k)-\fii(x))}-\frac{\fii'(\omega_k)\fii'(x)}{(\fii(\omega_k)-\fii(x))^2}\right)\\ & \qquad \qquad  \qquad   +\left(\frac{\fii'(\omega_k)\fii'(x)}{(\fii(\omega_k)-\fii(x))^2}-\frac{1}{(x-\omega_k)^2}\right)
.
\end{align*}
  The first summand converges to $\fii'(\omega_k)^2/3$ as $x\to \omega_k$ since 
  $$ \lim_{x\to 0} \left(\frac{1}{\sin^2(x)}-\frac{1}{x^2}\right)=\frac{1}{3}.$$
To deal with the second summand, we use the well-known formula
$$\lim_{x,y\to t}\left( \frac{\fii'(x)\fii'(x)}{(\fii(x)-\fii(y))^2}-\frac{1}{(x-y)^2}\right)=\frac{1}{6}S \fii(t).$$
\edem

It follows from Lemma \ref{alpha-lemma} that  
for almost every  $x \in \R$, there exists a   sequence $\{\omega_n\} \subset \R$ such that $\{ k_{\omega_n} \}$ is an orthonormal basis for $H(E)$ and $x=\omega_k$ for some $k\in\Z$. Hence, for such an $x$, we fix the corresponding sequence $\{\omega_n\}$ and use the notation $x=\omega_k$. 

Our first step is to rewrite the variational formulation of the Bergman metric in the following way:
\begin{align}
\rho_1(\omega_k)^2&=\frac{\sup \{ |h'(\omega_k)|^2 : \|h\|_{H(E)}=1,\;h(\omega_k)=0  \}}{ \pi^2 K(\omega_k,\omega_k)}\nonumber\\
&=\frac{1}{\pi\fii'(\omega_k)}\sup\left\{ \left|\frac{h'(\omega_k)}{E(\omega_k)}\right|^2 : \|h\|_{H(E)}=1,\;h(\omega_k)=0 \right\}\nonumber\\
&=\frac{1}{\pi\fii'(\omega_k)}\sup\left\{ \left|\left(\frac{h}{E}\right)'(\omega_k)\right|^2 : \|h\|_{H(E)}=1,\;h(\omega_k)=0 \right\}.\label{rho en prop}
\end{align}
Let $h\in H(E)$ satisfy 
$\|h\|_{H(E)}=1$ and $h(\omega_k)=0$.
Since the functions
$$
k_{\omega_n}(x)=\frac{|E(x)|\sin\big(\fii(x)-\fii(\omega_n)\big)}{\sqrt{\pi\fii'(\omega_n)}(x-\omega_n)},
$$
form an orthonormal basis for $H(E)$, there exists a sequence $\{c_n\}\in\ell^2(\Z)$ such that
\begin{align*}
\frac{h(x)}{E(x)}=\sum_{n\neq k} c_n \frac{k_{\omega_n}(x)}{E(x)}
&=\sum_{n\neq k} c_n \frac{\sin\big(\fii(x)-\fii(\omega_n)\big)}{\sqrt{\pi\fii'(\omega_n)}(x-\omega_n)\e^{-\im\fii(x)}}\\
&=\frac{1}{2\sqrt{\pi}}\sum_{n\neq k} c_n \Big(i\e^{\im \alpha}(-1)^n\Big) \frac{(1-\e^{2i(\fii(x)-\alpha)})}{\sqrt{\fii'(\omega_n)}(x-\omega_n)}.
\end{align*}
As we may assume that the sequence ${c_n}$ is finite, we can   differentiate term-by-term to get 
\begin{align*}
\left(\frac{h}{E}\right)'(x)
&=\frac{1}{\sqrt{\pi}}\sum_{n\neq k} \tilde{c}_n 
\left(\frac{\fii'(x)}{\sqrt{\fii'(\omega_n)}(x-\omega_n)}\right) -
\sum_{n\neq k} c_n 
\left(\frac{k_{\omega_n}(x)}{E(\omega_k)(x-\omega_n)}\right),
\end{align*}
where $\tilde{c}_n=c_n(\e^{\im \alpha}(-1)^n)$. Since $k_{\omega_n}(\omega_k)=0$ for every $n\neq k$, it holds that
\begin{align*}
\left(\frac{h}{E}\right)'(\omega_k)&=\frac{1}{\sqrt{\pi}}\sum_{n\neq k} \tilde{c}_n 
\left(\frac{\fii'(\omega_k)}{\sqrt{\fii'(\omega_n)}(\omega_k-\omega_n)}\right). 
\end{align*}
Therefore, the supremum in \eqref{rho en prop} can be rewritten as
\begin{align}
\rho_1(\omega_k)^2
&=\frac{\fii'(\omega_k)}{\pi^2}\sup\left\{ \left|\sum_{n\neq k} 
\frac{d_n}{\sqrt{\fii'(\omega_n)}(\omega_k-\omega_n)}\right|^2 :\ d_k=0, \ \ \sum_{n \in \Z}|d_n|^2=1 \right\}.\label{rho en prop 2}
\end{align}
As this is the dual-formulation of an $\ell^2$-norm, it follows immediately that
$$
\rho_1(\omega_k)^2=\frac{\fii'(\omega_k)}{\pi^2}\sum_{n\neq k} \frac{1}{\fii'(\omega_n)(\omega_k-\omega_n)^2}\,.
$$
By Lemma \ref{serie 2}, we get Theorem \ref{first intensity}.

\begin{rem}			\label{remark_distances}
%
  We observe that, with the notation above, the relation between the Bergman and 
  the  sine distances (see \cite[Proposition 8]{ARSW11}) gives that for any interval $I\subset \R$ with $\phi'$-measure equal to $1$, i.e., $\int_I \phi'(t) \dif t =1$, then
$$\int_I \rho_1(x) \dif x\gtrsim 1.$$ 
  Therefore, the expected number of real zeros in an interval of bounded $\phi'$-measure is bounded below.
  But these two integrals are in general not equivalent, in the sense that  there exist $I$ of $\phi'$-measure 1 with
$\int_I \rho_1(x) \dif x=+\infty$. Indeed, this happens with a half-line of $\phi'$-measure $1$ in the Bessel and Airy case (see Remark \ref{breaking down}). However, we have not been able to find an example of a 
bounded interval where this happens.    
\end{rem}

\section{Examples} \label{example section}

\subsection{The Paley-Wiener space} 
  As we mentioned in the introduction, the main example of a de Branges space is the Paley-Wiener space, 
  which corresponds to the choice $E(z)=\e^{-\im\pi z}$. We recall that the Paley-Wiener GAF is given by
$$\sum_n a_n \frac{\sin \pi (x-n)}{\pi (x-n)},$$
  where $a_n$ are real i.i.d.\,standard normal random variables. The phase function is $\phi(x)=\pi x$ and trivially
\begin{equation}\label{paleywiener}
	\rho_1(x) \simeq \phi'(x).
\end{equation}

\subsection{de Branges spaces with doubling phase function}
Here, we consider   de Branges spaces with phase function $\phi(x)$ such that the measure  $\phi'(x) \dif x$ is locally doubling on $\R$. 

We recall that a measure $\mu$ is doubling on $\R$ if there exists a   constant $C>0$ 
such that, for all intervals $I \subset \R$, the inequality $\mu(2I) \leq C \mu(I)$ holds. The measure is called locally doubling if it is doubling
for all intervals $I \subset \R$ satisfying  $\mu(I)\le 1$.

Locally   doubling de Branges spaces   are in some sense close to   Paley-Wiener spaces and are amenable to study. E.g.,  it is  known that these spaces correspond (cf.\,Remark \ref{relationModel}) to so called meromorphic one-component model subspaces of the Hardy space, and, in \cite{MNO}, density-type results for sampling and interpolating sequences, analogue to those known to hold for Paley-Wiener spaces, were obtained.

We need the following result, which is Lemma 2.3 in \cite{MNO}.

\begin{prop}						
  Let $H(E)$ be a de Branges space with phase function $\phi$. If the measure $\phi'(x)\dif x$   
   is locally doubling on $\R$, then there exist constants such that
\begin{equation}						\label{phiprimeequivalent}
 |\phi(s)-\phi(t)|\le 1\quad \implies \quad \phi'(s)\simeq \phi'(t).
\end{equation}
\end{prop}

As a generalisation of \eqref{paleywiener}, we obtain the following.

\begin{prop} \label{prutt}
  Let $H(E)$ be a de Branges space with phase function $\phi$. If the     measure $\phi'(x)\dif x$
   is locally doubling on $\R$,  then $$\rho_1(x)\simeq \phi'(x).$$
\end{prop}
 Before we give the proof of the proposition, we note that by the definition of the first intensity function, we immediately get the following  corollary.
\begin{cor}  \label{corollary}
 Let $H(E)$ be a de Branges space with phase function $\phi$. If the     measure $\phi'(x)\dif x$
   is locally doubling on $\R$, then, for $I\subset \R$, the de Branges GAF
$$F(x)=\sum_n a_n k_{\omega_n}(x)$$
  satisfies
$$\mathbb{E}[\# (Z_\R (F)\cap I   )]\simeq \int_I \phi'(x) \dif x.$$
\end{cor}
\begin{proof}[Proof of proposition \ref{prutt}]
	   Define the function 
$$
\Phi(z)=
\left\{
\begin{array}{ll}
\log |E(z)|  & \mbox{if } \Im z\ge 0, \\
\log |E^*(z)| & \mbox{if } \Im z<0.
\end{array}
\right.
$$
   Let $f\in H(E)$ be such that $f(x)=0$ and $\| f \|_{H(E)}=1.$ The Bernstein inequality (see, e.g., Lemma 2.10 in \cite{MNO}), says that
\begin{align*} 
	\frac{|f'(x)|^2}{K(x,x)} &=\frac{\pi}{\phi'(x)}\left| \left( \frac{f}{E} \right)'(x) \right|^2   &\lesssim
\phi'(x)^3 \int_{D_\phi(x,1)}  
|f(z)|^2 \e^{-2\Phi(z)}\dif m(z),
\end{align*}
where $D_\phi(x,1)$ is the disc in $\C$, having both diameter and intersection with $\R$ equal to $\{t : \abs{\phi(x)-\phi(t)}<1 \}$.
It is not difficult to show that
the measure $\phi'(\Re z)\chi_{D_\phi(x,1)}$   is a Carleson measure for   the Hardy spaces $H^2(\C_+)$ and $H^2(\C_-)$,
see \cite[Remark 2.12]{MNO}. Therefore
$$\phi'(x) \int_{D_\phi(x,1)}|f(z)|^2 \e^{-2\Phi(z)}\dif m(z)\lesssim \| f \|_{H(E)}^2=1.$$
Combining the previous two computations, we obtain
\begin{equation*}
	\frac{|f'(x)|^2}{K(x,x)} \lesssim \phi'(x)^2.
\end{equation*}
Taking into consideration the Bergman metric expression for the first intensity function, it follows immediately that  $\rho_1(x)\lesssim \phi'(x).$

  For the opposite inequality, let $x,y\in \R$ be such that $\phi(x)-\phi(y)=\pi,$ and set $f(z) = k_y(z)$. Clearly $f(x) = 0$ and $\| f \|_{H(E)}=1$.
Since
$$| {k}'_{y}(x)|^2=\frac{\phi'(x)^2 |E(x)|^2}{\pi  \phi'(y)|x-y|^2 },$$
we deduce, again from the Bergman metric expression for the first intensity, that 
$$\rho_1(x)^2\gtrsim \frac{|k'_{y}(x)|^2}{K(x,x)}=\frac{\phi'(x)}{\pi \phi'(y)|x-y|^2}\simeq \phi'(x)^2.$$
  In the last step, we use  (\ref{phiprimeequivalent}) and the relation $\phi'(c)(x-y) =\phi(x)-\phi(y)$ given by the mean value theorem.  
\end{proof}

 \subsection{The Airy kernel}
 The Airy function $\Ai(z)$ is the solution of the differential equation $$u''=z u$$ that,  for $t \in \R$, is expressed by the oscillatory integral 
$$\Ai(t)= \frac{1}{2\pi} \int_{-\infty}^{+\infty} \e^{\im tx+ \im\frac{x^3}{3}}\dif x.$$ For information about the Airy function one can 
consult the monograph \cite{SV2010}.

 The  function $E(z)=\Ai'(z)-\im \Ai(z) =: A(z)-\im B(z)$  is in the Hermite-Biehler class and has no real zeros. It therefore yields a de Branges space $H(E)$ with 
    reproducing kernel   given, for real $x\neq y$, by
\begin{align}						\label{airykernel}
	K(x,y) &=\frac{B(x)A(y)-B(y)A(x)}{\pi(x-y)} \nonumber \\ &=\frac{\Ai(x)\Ai'(y)-\Ai'(x)\Ai(y)}{\pi(x-y)}.
\end{align}

  The functions $\Ai(z),\Ai'(z)$ have zeros only on the negative real axis. We denote the zeroes of $\Ai$, as is customary, by
$$0>a_1>a_2>\dots.$$
In the following proposition, we summarize some facts about the resulting Airy de Branges space.

\begin{prop}
	Let $H(E)$ be the de Branges space given by the function $$E=\Ai'(z)-\im \Ai(z).$$ Then, the associated de Branges GAF can be expressed as
	\begin{equation}  \label{airygaf}
		F(x)=\sum_{n=1}^{\infty} X_n \frac{(-1)^{n+1}\Ai(x)}{\sqrt{\pi}(x-a_n)},
	\end{equation}
	where $a_n$ are real i.i.d.\,standard normal variables,  and   has first intensity   
	\begin{equation} \label{rho1airy}
		\rho_1(x)^2 =  \frac{1}{\pi^2 } \left(\frac{2}{3}\frac{\Ai(x)\Ai'(x)}{x\Ai(x)^2 - \Ai'(x)^2}-\frac{\Ai^4(x)}{4(x\Ai(x)^2-\Ai'(x)^2)^2}-\frac{x}{3}\right). 
	\end{equation}
	 Moreover, if we denote by $\phi$ the phase function of $E$, the following holds:
	\begin{itemize}
		\item[(i)] The measure $\phi'(x) \dif x$ is not locally doubling on $\R$.
		\item[(ii)] The function $\phi'(x)$ satisfies the estimates
		\begin{equation*}
\phi'(x)\simeq \left\{
	\begin{array}{ll}
		\frac{1}{2x^{3/2}}+O(x^{-5/2})  & \mbox{as } x\to +\infty, \\
		\frac{1}{\pi}  \left( \sin^2 \left( \frac{2}{3}|x|^{3/2}-\frac{\pi}{4}\right)+\frac{ \cos^2\left( \frac{2}{3}|x|^{3/2}-\frac{\pi}{4}\right)  }{|x|}\right)^{-1} & \mbox{as } x\to -\infty.
	\end{array}
\right.
		\end{equation*}
		\item[(iii)] The first intensity function satisfies the estimates
		\begin{equation*}
\rho_1(x)\simeq \left\{
	\begin{array}{ll}
		\frac{1}{4 \pi x}+O(x^{-3/2})  & \mbox{as } x\to +\infty, \\
		\frac{1}{\pi} \sqrt{\frac{|x|}{3}}   & \mbox{as } x\to -\infty.
	\end{array}
\right.
		\end{equation*}
	\end{itemize}
\end{prop}

One can compare this estimate with the corresponding one for the the determinantal case (\cite[(1.17)]{Soshnikov2000b})
   where the first intensity behaves as $\pi^{-2} \sqrt{|x|}$ when $x\to -\infty$ and decreases faster than exponentially when $x\to +\infty.$

\begin{rem} \label{breaking down}
If we compare $\phi'$ with $\rho_1$, we observe the following. For $x<0$,  the   first intensity function  behaves, in average,   like the oscillating function $\phi'(x).$ 
  Indeed, for $I \subset (-\infty,0)$,  it holds that
  \begin{equation*}
	\int_I \phi'(x) \dif x =1 \implies \int_I \rho_1(x) \dif x \simeq 1
	\end{equation*}
	To see this, simply use the well-known estimate (see \cite[formula (2.52)]{SV2010}),  
\begin{equation*} \label{airyzero}
	a_n=-\left( \frac{3\pi n}{2} \right)^{2/3}\Big(1+o(1)\Big).
\end{equation*}
	to see that
	$$\int_{a_{k+1}}^{a_k}\rho_1(x)\dif x \simeq \int_{k^{2/3}}^{(k+1)^{2/3}}  \sqrt{x}\dif x\simeq 1.$$
	For $x>0$ the situation is quite different. Indeed, the  analogy that can be seen when taking averages   breaks down for $x>0$, since,     unlike in   the determinantal case, the expected number of zeros on $(0,\infty)$ is unbounded while it has finite $\phi'$-measure. Indeed, the expected number of zeroes on  the interval $(0,N)$ grows as $\log N$ (see Remark \ref{remark_distances}). 
\end{rem}


\begin{figure}
    \begin{center}
        \begin{minipage}[c]{0.5\textwidth}
           \centering
           \includegraphics[width=0.9\textwidth]{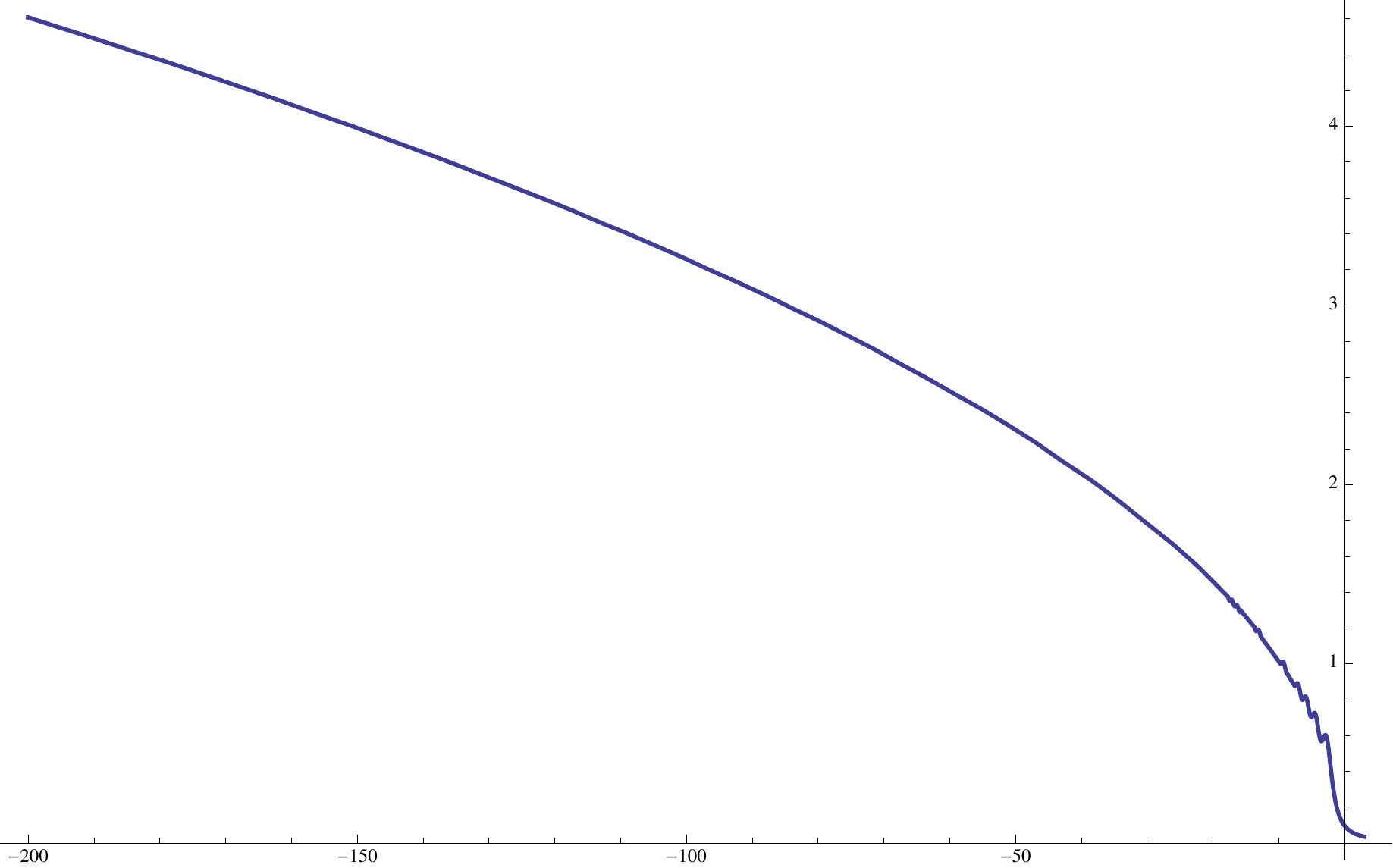}
           \par\vspace{0cm}
           \caption{First correlation function, $\rho_1(x)$.}
        \end{minipage}%
        \begin{minipage}[c]{0.5\textwidth}
           \centering
           \includegraphics[width=0.93\textwidth]{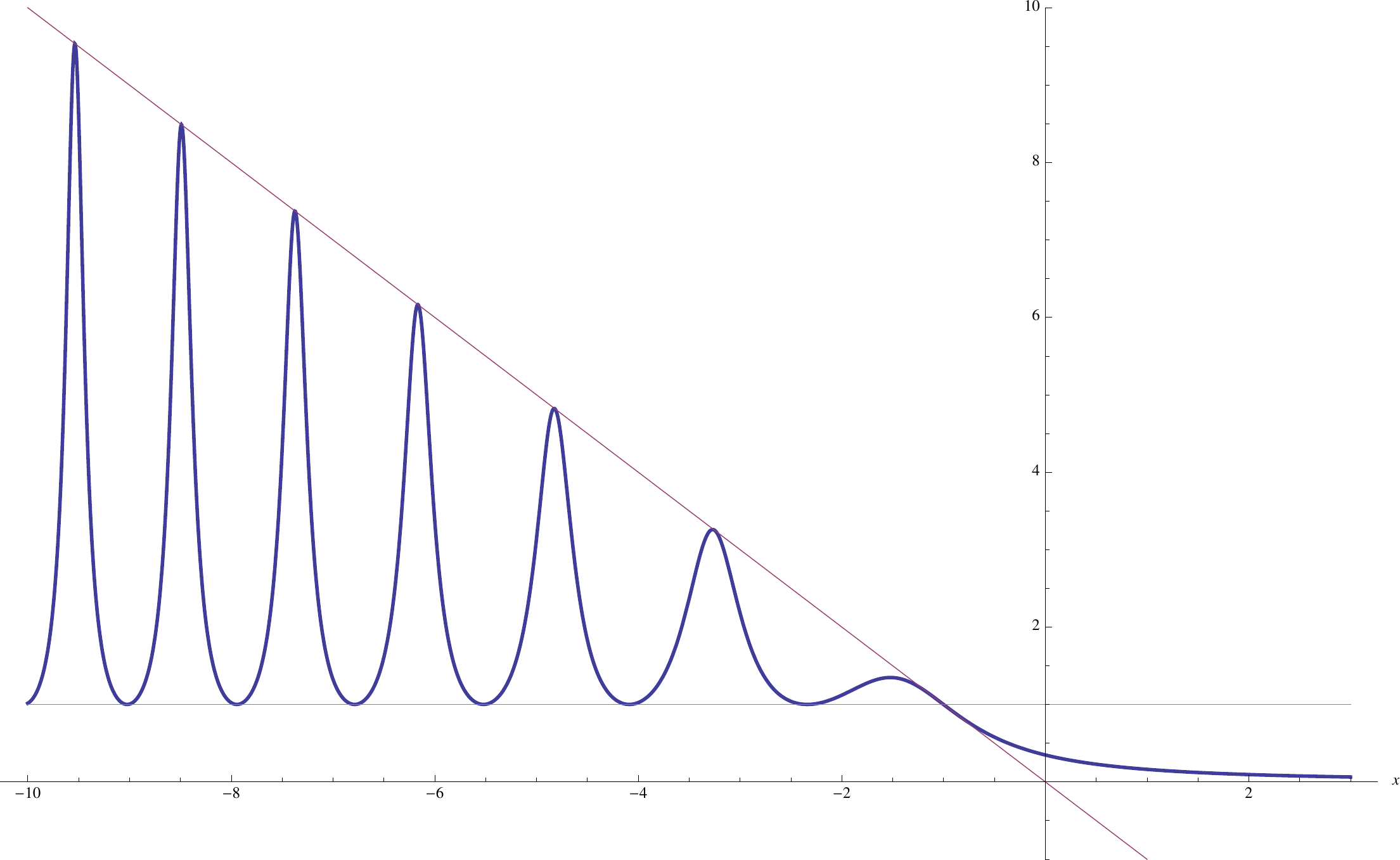}
           \par\vspace{0cm}
           \caption{Derivative of the phase function, $\phi'(x)$.}
        \end{minipage}%
    \end{center}
\end{figure}

%
%

\begin{proof}
%
Since $\Ai$ and $\Ai'$ do not have any zeroes in common, it follows that   $E$ is real exactly on the points, $a_n$, and so, these form a sequence of points which are exactly at  $\phi'$-distance $\pi$ from each other.  
Since $|\Ai(x)/E(x)|^2$ is not integrable (that it is not integrable on $(0,\infty)$ frollows from  \cite[items 10.4.59 and 10.4.61, p.\,448]{AbramovitzStegun64}   or \cite[formulas (2.44) and (2.45)]{SV2010}), by Lemma \ref{alpha-lemma}, the  normalized reproducing kernels $\{ k(x,a_n) \}_{n\in \N}$
  yield an orthonormal basis of $H(E)$.

 To obtain an explicit formula for the Airy de Branges GAF, we note that  from the relation  $\Ai''(x)=x\Ai(x)$, we get  
\begin{equation}
\begin{aligned} \label{phi-prime-airy}
\phi'(x)=\pi\frac{K(x,x)}{|E(x)|^2} &=\frac{A(x)B'(x)-A'(x)B(x)}{A^2(x)+B^2(x)} \\ &= \frac{\Ai'(x)^2-x \Ai(x)^2}{\Ai'(x)^2+ \Ai(x)^2}.
\end{aligned}
\end{equation}
Taking into account that  $\Ai(a_n)=0$ and   $\Ai(x)$ is real-valued on $\R$,  it is now straight-forward to see that in the Airy case,   formula \eqref{de Branges GAF} yields \eqref{airygaf}.

To compute $\rho_1$ for the Airy de Branges GAF, it is practical to use the following representation for the reproducing kernel
$$K(x,y)=\frac{1}{\pi}\int_0^{+\infty} \Ai(x+t) \Ai(y+t) \dif t,$$
that follows from \cite[formula (3.52)]{SV2010} and the fact that both $\Ai(x)$ and $\Ai'(x)$ decay as $x \rightarrow +\infty$. Combined with the Edelman-Kostlan formula (\ref{edelman-kostlan})
and formulas for   primitives  of products involving the Airy function, \cite[Chapter 3]{SV2010}, it yields \eqref{rho1airy}.
We remark that one can get a similar formula for the first intensity function for any de Branges space considering the related
Sturm-Liouville (or Schr\"odinger) equation (see \cite{Rem2002,MakarovPoltoratsky}).

To see that the Airy space is not locally doubling, we verify that   (\ref{phiprimeequivalent}) does not hold. To this end, 
   denote the (negative) zeros of $\Ai'(x)$ by $b_k$. Then (see \cite[p. 15]{SV2010}), it suffices to observe that
$$|\phi(b_k)-\phi(a_k)|=\frac{1}{2},\;\;\;\mbox{and}\;\;\; \phi'(a_k)=\frac{1}{\pi},\; \phi'(b_k)=\frac{|b_k|}{\pi}\simeq k^{2/3}.$$

More generally,  one may translate asymptotic  estimates for the Airy function and its derivative
in combination with \eqref{phi-prime-airy} and \eqref{rho1airy} to check that (ii) and (iii) holds, respectively. More specifically, in both cases, one uses the estimates   \cite[10.4.59 and 10.4.61]{AbramovitzStegun64} to obtain the asymptotics as $x \rightarrow + \infty$ and the estimates \cite[10.4.60 and 10.4.62]{AbramovitzStegun64} to obtain the asymptotics as $x \rightarrow - \infty$.

\end{proof}



\subsection{The Bessel kernel}

 For $\nu \geq -1/2$, the Bessel function of order $\nu$, denoted by $J_\nu$, is the solution to the differential equation
\begin{equation}			\label{besseldifequation}
  x^2 y''+xy'+(x^2-\nu^2)y=0,
\end{equation}
that can be expressed, using   Taylor series, as 
$$J_\nu(z)=\sum_{k=0}^\infty \frac{(-1)^k(\frac{z}{2})^{2k+\nu}}{k!\Gamma(k+\nu+1)}.$$ 
  
One can define a de Branges space   using the Hermite-Biehler function (see \cite{MakarovPoltoratsky}) 
 \begin{align*}
 	E_{\nu}(z) &= A_\nu (z)-\im B_\nu (z) \\ &= z^{-\nu/2}\Big(z^{1/2}J'_\nu(\sqrt{z})-\im J_\nu(\sqrt{z})\Big).
\end{align*}
 Observe that with the term $z^{-\nu/2}$ in the definition, the function $E_\nu(z)$ is entire. 
   The reproducing kernel for the de Branges space $H(E_\nu)$ is given by
\begin{equation}						\label{besselkernel}
	K(x,y)=(xy)^{-\nu/2} \frac{\sqrt{y}J_\nu'(\sqrt{y})J_\nu (\sqrt{x})-\sqrt{x}J_\nu'(\sqrt{x})J_\nu(\sqrt{y})}{\pi(x-y)},\;\;\;x,y\in \R.
\end{equation}
    In the determinantal case, it is more common to define the kernel without the extra term $(xy)^{-\nu/2},$ but this version appears also in 
    the literature. See, e.g., \cite{Lubinsky_Universality_Hard_edge}.

	The functions $J_\nu$ and $J_\nu'$ only have zeroes on the positive half of the real line,
	and we denote the zeroes of $J_\nu$ by
	\begin{equation*}
		0< j_{\nu,1} < j_{\nu, 2} < \ldots.
	\end{equation*}
	We summarize some facts about the Bessel de Branges GAF in the following proposition.
  \begin{prop}
  	For $\nu \geq -1/2$, let $H(E_\nu)$ be the de Branges space given by the function $E_\nu$ defined above. Then
	the associated de Branges GAF can be expressed by
	\begin{equation}	\label{besselGAF}
		F(x)=\sqrt{\frac{2}{\pi}}x^{-\nu/2}\sum_{k=1}^{\infty} a_n \frac{(-1)^k j_{\nu,k} J_\nu(\sqrt{x})}{x-j_{\nu,k}^2},
	\end{equation}
	where $a_n$ are real i.i.d.\,standard normal variables
	
	 Moreover, if we denote by $\phi$ the phase function of $E_\nu$, the following holds:
	\begin{itemize}
		\item[(i)] The measure $\phi'(x) \dif x$ is not locally doubling on $\R$.
		\item[(ii)] The function $\phi'(x)$ satisfies the estimates 
			\begin{equation*}
\phi'(x)\simeq \left\{
	\begin{array}{ll}
		  \frac{1}{2}\frac{1}{(\sqrt{x}\sin t+c_2 \cos t)^2+\cos^2 t -2 c_1 \sin^2 t} +O(x^{-1/2}) & \mbox{as } x\to +\infty, \\[2mm]
		  \frac{|x|^{-3/2}}{2}  +O(x^{-2}) & \mbox{as } x\to -\infty,
	\end{array}
\right.
		\end{equation*}
	 where $t=\sqrt{x}-\frac{\nu \pi}{2}-\frac{\pi}{4}$, $\mu=4\nu^2$, $c_1=\frac{(\mu-1)(\mu+15)}{128}$ and $c_2=\frac{\mu+3}{8}$.
		\item[(iii)] The first intensity function satisfies the estimates
		\begin{equation*}
\rho_1(x)\simeq \left\{
	\begin{array}{ll}
		 \frac{1}{2\pi \sqrt{3 x}} & \mbox{as } x\to +\infty, \\
		\frac{1}{4 \pi |x|}+O(|x|^{-3/2}) & \mbox{as } x\to -\infty.
	\end{array}
\right.
\end{equation*}
	\end{itemize}

  \end{prop}

      Observe that the formula for $\phi'$ above says  that as   $x\to +\infty$, we have that   $2 \phi'(x)$ oscillates between $1/(x-2c_1)$ and the constant $1/(1+c_2^2).$

%
  
As in the Airy case, we can interpret the first intensity function as giving an ``average'' of the $\phi'$-function, since, for $k\ge 1$, we have
$$\int_{j_{\nu,k}^2}^{j_{\nu,k+1}^2} \rho_1(x) \dif x \simeq \int_{j_{\nu,k}^2}^{j_{\nu,k+1}^2} \phi'(x) \dif x=\pi.
$$
And, as in the Airy case, the analogy breaks for on the other half line:
$$
\pi=\int_{-\infty}^{j_{\nu,1}^2} \phi'(x) \dif x\le \int_{-\infty}^{j_{\nu,1}^2} \rho_1(x) \dif x=\infty,$$
  see Remark \ref{remark_distances}.

  Also, it is not difficult to estimate the first intensity for the the determinantal process given by the kernel $K(x,y)(xy)^{\nu/2}$
$$K(x,x) x^{\nu}\simeq \frac{1}{\pi  x^{1/2}}\left( 1+O(x^{-1/2})\right)$$
  for $t=\sqrt{x}-\frac{\nu \pi}{2}-\frac{\pi}{4},$ and we get a bigger amount of points (in expectation) than in the GAF case.


\begin{figure}
    \begin{center}
        \begin{minipage}[c]{0.5\textwidth}
           \centering
           \includegraphics[width=0.9\textwidth]{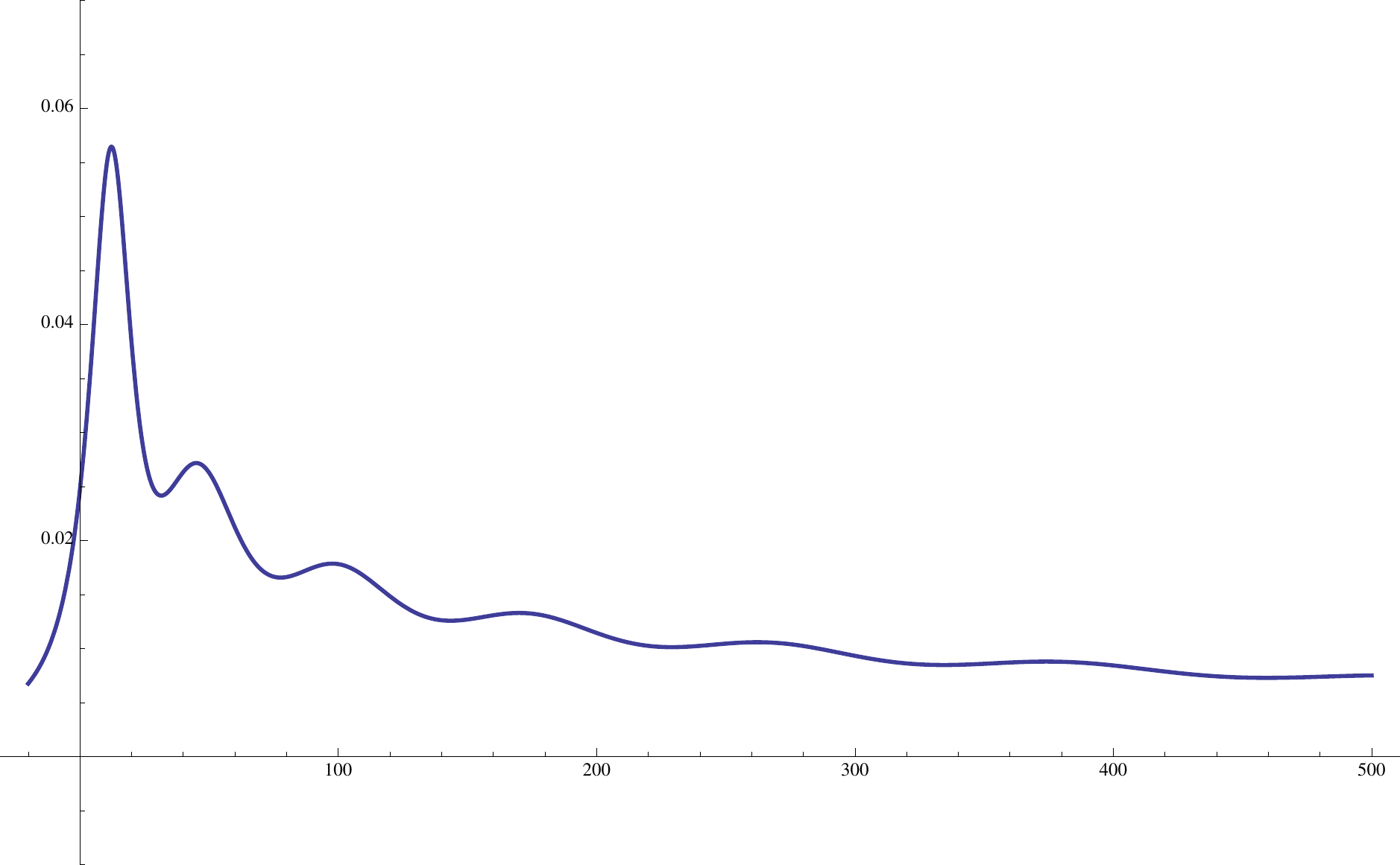}
           \par\vspace{0cm}
           \caption{First correlation function $\rho_1(x)$, for $\nu=\frac{1}{2}$.}
        \end{minipage}%
        \begin{minipage}[c]{0.5\textwidth}
           \centering
           \includegraphics[width=0.93\textwidth]{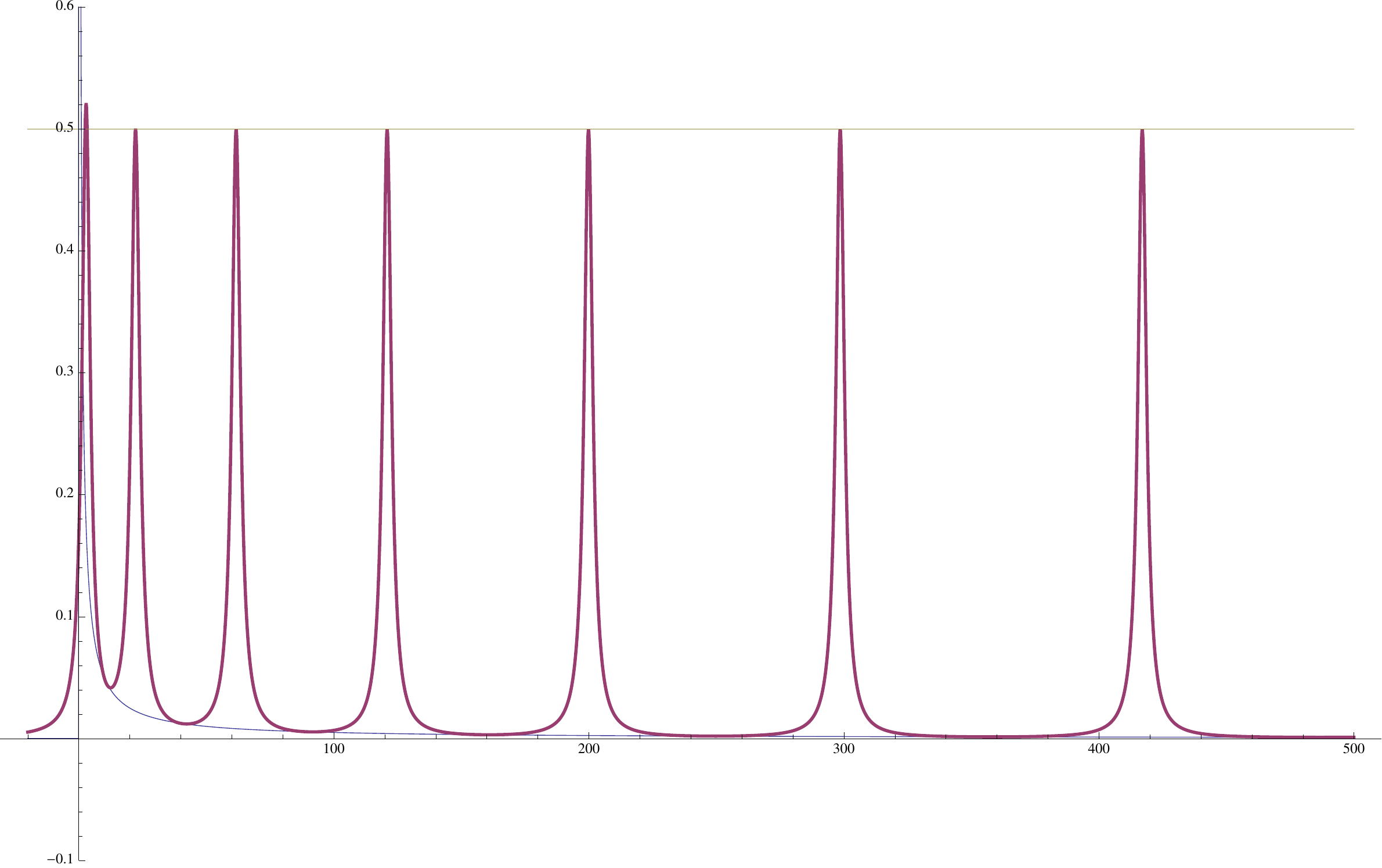}
           \par\vspace{0cm}
           \caption{Derivative of the phase function for $\nu=\frac{1}{2}$.}
           \label{bessel fi draw}
        \end{minipage}%
    \end{center}
\end{figure}

    \begin{proof}
    	Since $E_\nu$ is real exactly on the points, $j_{\nu,k}$, these are necessarily at $\phi'$-distance $\pi$ from each other.  
    As in the Airy case, by using the well-known
    asymptotic estimates for the Bessel function and its derivative (\cite[items 9.2.5 and 9.2.11, p. 364]{AbramovitzStegun64}),
  one can see that $|B_\nu(x)/E_\nu(x)|^2$ is not integrable and therefore, by Lemma \ref{alpha-lemma},  that $\{ K(x,j_{\nu,k}^2) \}_{k\ge 1}$ is an orthogonal basis in $H(E).$


A straight forward calculation yields that the Bessel de Branges GAF may be given by the formula \eqref{besselGAF}.

    By using Bessel's equation (\ref{besseldifequation}), one can show that 
\begin{equation} \label{phi-prime-bessel}
	\begin{aligned}
	\phi'(x) &= \frac{B'_\nu(x)A_\nu(x)-B_\nu(x)A'_\nu(x)}{A_\nu^2(x)+B_\nu^2(x)} \\
	&=\frac{1}{2}\frac{J'_\nu(\sqrt{x})^2 + \left(1-\frac{\nu^2}{x}\right)J_\nu(\sqrt{x})^2  }{xJ'_\nu(\sqrt{x})^2+J_\nu(\sqrt{x})^2}.
	\end{aligned}
\end{equation}

 From this, we can observe, as in the Airy case,   that the  de Branges spaces given by the Bessel functions are not locally doubling. Indeed, since the zeros of $J_\nu$ and $J'_\nu$ interlace (i.e., $j_{\nu,k}<j_{\nu,k+1}'<j_{\nu,k+1}$, where we let $j_{\nu,k}'$ denote the zeroes of $J_\nu'$), one can make the following argument.
 If $H(E_\nu)$ was locally doubling, then, by   property (\ref{phiprimeequivalent}), it would follow that
  $$\phi'(j_{\nu,k}^2)\simeq \phi'(j_{\nu,k}'^2).$$
From the expression of $\phi'$, we get
  $$\phi'(j_{\nu,k}^2)=\frac{1}{2}\frac{1}{j_{\nu,k}^2}\quad\text{and}\quad\phi'(j_{\nu,k}'^2)=-\frac{1}{2}\frac{J_\nu''(j_{\nu,k}')}{J_\nu(j_{\nu,k}')}.$$
Since  $J_\nu(x)$ satisfies Bessel's equation of order $\nu$, i.e., equation \eqref{besseldifequation},
	we get a contradiction from the limit 
$$\frac{J_\nu''(j_{\nu,k}')}{J_\nu(j_{\nu,k}')}=\frac{\nu^2 }{j_{\nu,k}'^2}-1\to -1,\;\;k\to \infty.$$

    Compared to the Airy case,   it is much more difficult to get asymptotic estimates for the phase function, its derivative, and the first intensity function. 

Using formula \eqref{phi-prime-bessel} and an accurate estimate for $J_\nu(x)$ and its derivative $J'_\nu(x)$ (see, e.g.,  \cite[items 9.2.5 and 9.2.11, p. 364]{AbramovitzStegun64}), we get the estimate for $\phi'(x)$ as $x \rightarrow + \infty$.
%
To estimate what happens when $x\to -\infty$, we combine \eqref{phi-prime-bessel} with   asymptotic estimates for the modified Bessel functions of the first kind $I_\nu(z)$ and its derivative (see, e.g.,   \cite[items 9.7.1 and 9.7.3, p. 377-378]{AbramovitzStegun64}). Here, we also need
 the relation
$$J_\nu(i z)=\e^{\frac{\nu i \pi}{2}}I_\nu(z),$$
  which is valid for $-\pi<\arg(z)\le \frac{\pi}{2},$ and so, in particular for $z=\sqrt{x}$ for $x>0.$
  %

 To obtain an explicit   expression for $\rho_1(x)$, in principle it is possible to proceed in the same way as for the Airy function, using the reproducing property \cite[p. 295]{TracyWidom1994b}
$$2\pi K(u,v)= (uv)^{-\nu/2} \int_0^1 J_\nu(\sqrt{ut}) J_\nu(\sqrt{vt})\dif t$$
  and the Edelman-Kostlan formula. However, this requires far more work than in the Airy case. Instead, it is possible to get the  estimate as $x \rightarrow +\infty$ using direct estimates on $\phi'$, while to get an estimate as $x\to -\infty$, one can combine direct estimates with Lemma \ref{schwarzian lemma}. 
%
    \end{proof}


\subsection{Orthogonal Polynomials}
The de Branges spaces above all belong to a more general class of spaces which can be built in connection with  
  singular or regular Sturm-Liouville problems. 
  More precisely,  the Airy equation on $[0,+\infty)$ and the Bessel equation in $(0,1]$
  are examples of singular Sturm-Liouville problem in the limit-point and in the circle-point case.
  For more details see \cite{Rem2002, MakarovPoltoratsky}. 
  A discrete version of these problem is to consider spaces of
  orthogonal polynomials.

  Let $\mu$ be a finite positive Borel measure $\mu$ on $\R$ whose support is not a finite number of points and which has finite moments of all orders, i.e.,
  \begin{equation*}
	\int_\R \abs{x}^n \dif \mu < \infty, \qquad \forall n\in\N_0.
\end{equation*}
    Denote by $p_n(x)=\gamma_n x^n+\dots,\;\;\gamma_n>0$, 
   a set of polynomials orthonormal with respect to the measure $\mu$, i.e., satisfying  
$$\int_{-\infty}^{+\infty} p_n(x)p_m(x)d\mu(x)=\delta_{n,m}.$$

  One can define the Hermite-Biehler function $E_n(z)=\sqrt{\pi \frac{\gamma_{n-1}}{\gamma_n}}(p_{n-1}(z)-ip_n (z)),$ and there is
a unitary equivalence between the de Branges space $H(E_n)$ and the space of orthonormal polynomials with respect to $\mu$.
  The associated de Branges GAF has the same distribution as the random polynomials
\begin{equation}				\label{polynomialGAF}
 P_n(x)=\sum_{k=0}^{n-1} a_k p_k(x)
\end{equation}
  where $a_k$ are real i.i.d.\,real standard normal random variables. This is clear because the covariance kernel of this 
  gaussian proces is given by 
  $K_n(x,y)=\sum_{k=0}^{n-1} p_k(x) p_k(y).$

  The problem of estimating asymptotically the expected number of real 
  zeros has been studied in some cases, see \cite{Wil97} for the Legendre case.
  
  In principle, by identifying the phase, it would be possible to compute the expected number of real 
  zeros, but this seems far from being trivial.
  
  One case which can be studied explicitly is the orthonormal polynomials with respect to the measure $\dif \mu = (x^2+a^2)^{-n} \dif x,$  
  where $a>0$ and $n\ge 2.$ In this case, one can take $E(z)=(z+\im a)^{n}.$ If $\phi(x)$ is a primitive of the Poisson kernel, i.e. $\phi'(x)=\frac{a}{x^2+a^2},$ 
	a straight-forward computation yields $S\phi=-2(\phi')^2$. From this it follows that a phase function for $H(E)$ is $n\phi(x)$ and
 $$\rho_1(x)=\frac{1}{\pi}\sqrt{\frac{n^2-1}{3}}\phi'(x).$$
Consequently, the expected number of real zeros of (\ref{polynomialGAF}) is, in this case,
$$\E[\#(Z_\R(P_n)\cap \R) ]=\int_\R \rho_1(x)=\sqrt{\frac{n^2-1}{3}}.$$


\section{Proofs of Theorems \ref{rigidity_1} and \ref{rigidity_2}}

We  begin with a   lemma on the Schwarzian derivative. It is one of four lemmas from which Theorem \ref{rigidity_2} follows.
\begin{lem}
 Suppose that, for all $x \in \R$, 
  $$S \phi_1(x)  + 2\phi_1'(x)^2=S \phi_2(x)+ 2\phi_2'(x)^2.$$
  If we put  $\psi=\phi_2 \circ \phi_1^{-1}$, then it holds that
\begin{equation}			\label{diffcomposition}
1=\psi'(t)^2+\frac{1}{2}S\psi(t).
\end{equation}
\end{lem}

\begin{proof}
  The Chain Rule for the Schwarzian derivative is
$$S(f\circ g)=(S f\circ g)(g')^2+ S g,$$
whence
$$S \phi_1( \phi_1^{-1}(t)) = - \frac{S (\phi_1^{-1})(t)}{(\phi_1^{-1})'(t)^2}.$$
	Combined with the usual chain rule, this yields
  $$2 \phi_1'\big(\phi_1^{-1}(t)\big)^2+S \phi_1\big(\phi_1^{-1}(t)\big)=
  \frac{2}{ (\phi_1^{-1})'(t)^2}-\frac{S (\phi_1^{-1})(t)}{(\phi_1^{-1})'(t)^2}.$$
  Applying the hypothesis to the left-hand side above, with $x=\phi_1^{-1}(t)$,  and rearranging the terms, this becomes
\begin{equation*}
 2  =2 \phi_2'\big(\phi_1^{-1}(t)\big)^2(\phi_1^{-1})'(t)^2+S \phi_2\big(\phi_1^{-1}(t)\big)(\phi_1^{-1})'(t)^2+S (\phi_1^{-1})(t).
\end{equation*}
Now, using both    chain rules on the right-hand side, we get
\begin{equation*} 
2 = 2 (\phi_2\circ \phi_1^{-1})'(t)^2+S (\phi_2 \circ \phi_1^{-1}) (t),
\end{equation*}
whence the result follows.
\end{proof}

Introducing the change of variables $x=\log \psi'$ and $y=x'$, we can turn \eqref{diffcomposition}    into a   second order differential equation. 
Indeed, 
\begin{equation}					\label{system}
 \left\{
	\begin{array}{ll}
		 {x}' & = y,  \\
		 {y}' & = 2-2\e^{2x}+\frac{1}{2}y^2.
	\end{array}
\right.
\end{equation}

To obtain a linear differential equation, in terms of   differentiation with respect to  $x$, we observe that, by the chain rule,
\begin{align*} 
\frac{\dif y}{\dif x} =  \frac{\dif y}{\dif t} \frac{\dif t}{ \dif x } &=  \frac{2-2\e^{2x}+\frac{1}{2}y^2}{y} \\ &=(2-2\e^{2x})\frac{1}{y}+\frac{1}{2}y. 
\end{align*}
It now follows,  by the further change of variables   $z=y^2$, that
$$
\frac{\dif z}{\dif x}=z+2(2-2\e^{2x}).
$$
This linear differential equation has  the general solution
$$
z(x)=C \e^x-4(1+\e^{2x}),
$$
whence, by substituting back, it follows that  if $\big(x(t),y(t)\big)$ is a solution of (\ref{system}), then
$$
y(t)^2=C \e^{x(t)}-4(1+\e^{2x(t)}).
$$
Therefore,   the system  has  periodic orbits, and,  by taking different values of $C\ge 8$, one gets the phase portrait in Figure \ref{phase_portrait}. 
 
\begin{figure}
  \includegraphics[height=15cm]{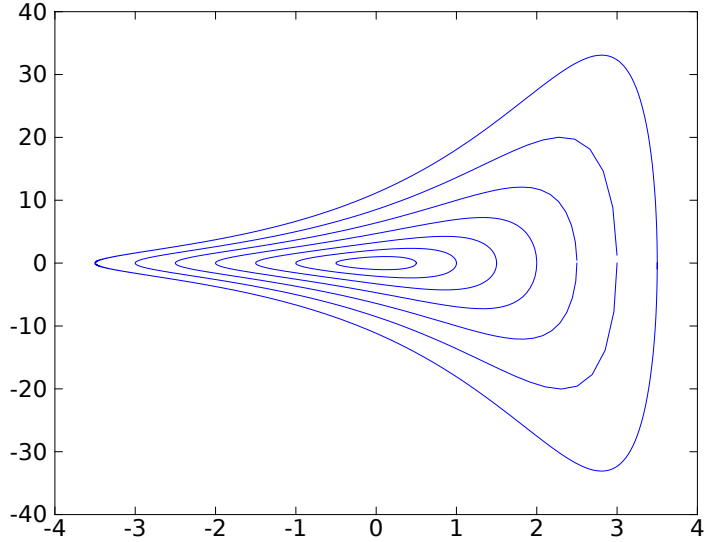}
  \vskip -4cm
   \caption{Phase portrait of \eqref{system}.}
  \label{phase_portrait}
\end{figure}  

The next result shows that the function $\psi'(t) = (\phi_2\circ \phi_1^{-1})'(t)$ is $\pi$-periodic.

\begin{lem}
	The solutions to the system (\ref{system}) have period $\pi$.
\end{lem}
\begin{proof}
  Let $(x,y)$ be a solution to (\ref{system}). As $x'=y$, it is clear that it is enough to show that $x(t)$ has period $\pi.$
  For $C>8$, let $x_\pm=x_\pm(C)$ be the two numbers such that $C\e^x-4(1+\e^{2x})=0$, i.e.,    the points where $y=0$ in the phase portrait.

  Let $T$ be the period of the solution $(x,y)$. If we let $x(0) = x_-$, then, by symmetry,  from time $0$ to time $T/2$ the function $x(t)$ goes from $x_-$ to $x_+$. 
  Therefore, by the inverse change of variable $x(t) = s$, we get
\begin{align*}
	\frac{T}{2}=\int_0^{T/2}\dif t&=\int_{x_-}^{x_+}\frac{ds}{\sqrt{C\e^s-4(1+\e^{2s})}} \\
	&=\int_{1/r_+}^{r_+}\frac{dr}{r\sqrt{Cr-4(1+r^2)}}\\	
&=\frac{1}{2} \left[ \arcsin \left( \frac{Cr-8}{r\sqrt{C^2-8^2}} \right) \right]_{1/r_+}^{r+}=\frac{\pi}{2}.
\end{align*}
  Here, we used, in the first line, that $$x'(t)=y(t)=\sqrt{C\e^{x(t)}-4(1+\e^{2x(t)})}=\sqrt{C\e^{s}-4(1+\e^{2s})},$$
  and, in the second line, the change of variables
  $\e^s=r$. Since $x_{\pm}$ corresponds to $y=0$, it follows that
$$r_\pm=\frac{1}{2}\left( \frac{C}{4}\pm \frac{\sqrt{C^2-8^2}}{4}\right),$$  and
    $x_++x_-=\log r_++\log r_-=\log \,r_+r_-=0$. The last step of the computation now follows.
\end{proof}

\begin{lem} \label{lemmalemma}
  Let $(x,y)$ be a solution of the system \eqref{system} such that  
  $y(0)=0$.
If
$$U(s)=\int_0^s \e^{x(t)}\dif t,$$
  then, for all $k \in \Z$, we have
$$U\left(k\frac{\pi}{2}\right)=k\frac{\pi}{2}.$$
\end{lem}
\begin{proof}
By essentially the same computations as in the previous proof, we get
\begin{align*}
	U(\pi/2) = \int_0^{\pi/2} \e^{x(t)} \dif t  &= \int_{1/r_+}^{r_+} \frac{dr}{\sqrt{Cr-4(1+r^2)}} \\ &=\int_{1/r_+}^{r_+} \frac{du}{u\sqrt{Cu-4(1+u^2)}}=\frac{\pi}{2}.
\end{align*}
Here, the only difference is the change of variables $u=1/r$ in the second line.

With this,   it is enough to show that, for $0<t<\pi$,
$$
x(t)=x(\pi-t).
$$
  Indeed, defining 
$$
(\widetilde{x}(t),\widetilde{y}(t))=(x(\pi-t),-x'(\pi-t)),
$$ 
we see that  $(\widetilde{x}, \widetilde{y})$ is a solution of \eqref{system} with initial condition  $(\widetilde{x}(0),\widetilde{y}(0))=(x_0,0)$. Therefore,   it follows by uniqueness that
$$
(x(t),y(t))=(x(\pi-t),-x'(\pi-t)),
$$
whence the desired conclusion follows.
\end{proof}

The above result shows that if  $(\phi_2\circ \phi_1^{-1})''(0)=0$, which is tantamount to $y(0)=0$, 
then the function
$$
s \longmapsto \int_0^s (\phi_2\circ \phi_1^{-1})'(t)\dif t
$$ 
is $\pi$-periodic.

Let $\alpha\in \R$ be such that $(\phi_2\circ \phi_1^{-1})''(\alpha)=0$. Indeed, such an $\alpha$ exists 
because $(\log \psi',\psi''/\psi')$ for $\psi=\phi_2\circ \phi_1^{-1}$ is a solution of the system (\ref{system}).
Now, we observe that the solution $(x,y)$ to \eqref{system}, where $x(t)=\log (\phi_2\circ \phi_1^{-1})'(t+\alpha),$ satisfies 
$x'(0)=0.$ Hence, by Lemma \ref{lemmalemma} above,
$$
(\phi_2\circ \phi_1^{-1})\left(\frac{\pi k}{2}+\alpha \right)-(\phi_2\circ \phi_1^{-1})(\alpha)=\frac{\pi k}{2}.
$$
So, for $v_k=\phi_1^{-1}(\frac{\pi k}{2}+\alpha)$, we get
$$
\phi_2(v_k)-\phi_2(v_0)=\phi_1(v_k)-\phi_1(v_0).
$$
Therefore,   we have that 
$$
\phi_2(v_k)=\phi_1(v_k)+\beta,
$$
where $\{v_k\}$ are the points such that $\phi_1(v_k)=\alpha \, (\operatorname{mod} \frac{\pi}{2}),$ and $\beta \in \R$ is a constant.
  
The proof of Theorem \ref{rigidity_2} is now a consequence of the following modification of a result of de Branges \cite[Theorem 24]{dB1968}, which is well-known to specialists.
\begin{lem}
Let $E_1$ and $E_2$ be two Hermite-Biehler functions without real zeroes, with phase functions $\phi_1$ and $\phi_2$, respectively.  Suppose that there exist   $\alpha,\beta\in\R$ so that
$$
\phi_2(t)=\phi_1(t)+\beta
$$
for all $t \in \R$ such that $\phi_1(t)=\alpha \, (\operatorname{mod} \frac{\pi}{2})$. Then, there exists a non vanishing real entire function $S$ such that $F(z)\mapsto S(z)F(z)$ is an isometric isomorphism from $H(E_1)$ onto $H(E_2)$.
\end{lem}
\begin{proof}
Define the Hermite Biehler functions $\widetilde{E}_1$ and $\widetilde{E}_2$ by
$$
\widetilde{E}_1(z)=\e^{-\im(\beta-\alpha)}E_1(z)\,,\peso{and} \widetilde{E}_2(z)=\e^{\im \alpha}E_2(z).
$$
If $\widetilde{\phi}_1$, and $\widetilde{\phi}_2$ denote their corresponding phase functions, then
$$
\widetilde\phi_1(t)=\widetilde{\phi}_2(t),
$$
whenever $t$ is such that $\widetilde\phi_2(t)=0 \, (\operatorname{mod} \frac{\pi}{2})$. Therefore, by Theorem 24 in \cite{dB1968}, there exists a real entire function $ {S}(z)$  such that
$$
F(z)\mapsto  {S}(z)F(z)
$$
is an isometric map between the spaces $H(\widetilde{E}_1)$ and $H(\widetilde{E}_2)$. Since $H(E_1) = H(\widetilde{E_1})$ and $H(E_2) = H(\widetilde{E_2})$, in the sense of Hilbert spaces, it follows that $S$ induces an isometry between the original spaces $H(E_1)$ and $H(E_2)$. A priori, this function may have real zeroes. However, since  neither $E_1$ nor $E_2$ have real zeros, the function $ {S}$ never vanishes. Indeed,   the function $ {S}$ satisfies the identity
$$
K_2(z,w)= {S}(z)K_1(z,w)\overline{ {S}(w)},
$$
and $K_j(x,x)=\frac{1}{\pi}\phi'_j(x)|E_j(x)|^2$ for $j=1,2$ and $x\in\R$. From this it follows that, for every $x \in \R$,
$$
| {S}(x)|=\frac{K_2(x,x)}{K_1(x,x)}=\frac{\phi'_2(x)|E_2(x)|^2}{\phi'_1(x)|E_1(x)|^2}\neq 0.
$$
From this we conclude that $ {S}$ is   an isometric isomorphism between the spaces $H( {E}_1)$ and $H( {E}_2)$.
\end{proof}

Now we can easily deduce Theorem \ref{rigidity_1} from Theorem \ref{rigidity_2}.

\proof(Theorem \ref{rigidity_1})
  Let $F(x)$ be a de Branges GAF defined by the reproducing kernel $K_1(z,w)$ of the space $H(E_1)$ and
  let $G(x)$ be the de Branges GAF defined by the reproducing kernel $K_2(z,w)$ of the space $H(E_2).$
  It follows from Theorem \ref{rigidity_2}
  that $S(x)F(x)$ and $G(x)$ have the same covariance kernel and therefore, as they are gaussian processes, $S(x)F(x)$ and $G(x)$ have the same
  distribution, but $S(x)$ does not vanish, so we get the result.

\qed




\bibliographystyle{amsplain}

\providecommand{\bysame}{\leavevmode\hbox to3em{\hrulefill}\thinspace}
\providecommand{\MR}{\relax\ifhmode\unskip\space\fi MR }
\providecommand{\MRhref}[2]{%
  \href{http://www.ams.org/mathscinet-getitem?mr=#1}{#2}
}
\providecommand{\href}[2]{#2}

\end{document}